\documentclass{article}

\usepackage{latexsym,amsmath,amsthm,amssymb}
\usepackage{mathrsfs,wasysym,graphicx,lscape}
\usepackage{esint}
\usepackage{color}
\begin{document}

\newtheorem{theorem}{\indent Theorem}[section]
\newtheorem{proposition}[theorem]{\indent Proposition}
\newtheorem{definition}[theorem]{\indent Definition}
\newtheorem{lemma}[theorem]{\indent Lemma}
\newtheorem{remark}[theorem]{\indent Remark}
\newtheorem{corollary}[theorem]{\indent Corollary}

\title{\textbf{Uniform Exponential Attractors for Non-Autonomous Dynamical Systems and an Application to the damped hyperbolic equation with sub-quintic growth}}

\author{ Qingquan Chang$^{1}$\footnote{changqingquan@lzu.edu.cn} \and Dandan Li$^2$ \footnote{Corresponding Author: dandanli@gbu.edu.cn} }
\date{}
\maketitle

\begin{center} 
       {\small 1.  School of mathematics and statistics, Lanzhou  University, Lanzhou, Gansu, 730000, P.R. China}\\
      {\small 2. Department of Mathematics, School of Science, Great Bay University, Dongguan, Guangdong, 523000, P.R. China} 
\end{center}

\renewcommand{\theequation}{\arabic{section}.\arabic{equation}}
\numberwithin{equation}{section}

\begin{abstract}
In this paper, we develop an abstract framework for the theory of uniform exponential attractors in non-autonomous dynamical systems. Our construction employs time-dependent absorbing sets while ensuring exponential attraction among the constructed sets. Moreover, we show that the resulting exponential attractors are uniform—that is, they remain stable under perturbations.
As an application of this abstract theory, we study the weakly damped  hyperbolic equation with sub-quintic nonlinearity. Using the proposed approach, we construct exponential attractors for the process generated by the Shatah–Struwe solution, which satisfies both energy and Strichartz estimates. Furthermore, we establish the Hölder continuity of these exponential attractors with respect to the external forcing term.
\end{abstract}
{\bf Keywords}: Uniformity exponential attractors; Non-autonomous dynamical systems; Weakly damped  hyperbolic equation; sub-quintic nonlinearity; Shatah-Struwe solution.

\footnote[0]{\hspace*{-7.4mm}
AMS Subject Classification: 35B40; 35B45; 37L15; 37L30\\

}

\section{Introduction}
\noindent

The evolution differential equations are often used to  describe and simulate concrete phenomenons which arise in mechanics and physics. But many equations cannot be solved analytically. The absence of analytical solutions for evolution equations makes it impossible to theoretically explain new phenomena appearing in new areas of science and technology.
{ By borrowing ideas and mathematical tools from dynamical systems to conduct qualitative analysis of partial differential equations, the discipline of infinite-dimensional dynamical systems emerged.}
Infinite-dimensional dynamical systems, with asymptotic analysis as the primary method, could provide proper guidelines for numerical simulation and analysis in the nonlinear evolution systems.   


According to the theory of infinite-dimensional dynamical systems, the long-term behavior of a dissipative dynamical system can be characterized by the presence of a complex attractor toward which all trajectories converge as \( t \to \infty \) (see \cite{Babin1992,Temam1998}). This attractor may exhibit a fractal structure, such as a Cantor set or a product of a Cantor set with an interval. It thus serves as the fundamental mathematical model for describing the observed nonstationary flow.
Global attractor is the first and the most popular concept to describe the long-time behavior of infinite-dimensional dynamical systems.
 However, the global attractor is only valid for autonomous systems, and at the same time it is very sensitive to disturbances, and its attracting speed may be particularly slow. These two defects make it difficult for the global attractor to be applied in detail mathematical physics problems. In 1988, Foias, Sell and Temam constructed inertial manifolds to achieve long-term behavior research and dimensionality reduction in \cite{Foias1988}. However, there are extremely harsh conditions for obtaining an inertial manifold, because its existence usually depends on the spatial dimension, the spectral gap condition of the operator, and even the shape of the domain. For many important equations, such as the two-dimensional Navier-Stokes equation, the existence of the inertial manifold is still an open problem, see \cite{Constantin1989a,Constantin1989b,Mallet-Paret1988,zelik2004}.

To overcome the shortcomings of global attractors, Eden \emph{et al} proposed the concept of exponential attractor in \cite{Eden1994} as an intermediate object between the notion of the global attractor and the inertial manifold. The exponential attractor is a compact set with positive invariance, finite fractal dimension, and exponential attraction. The exponential attraction property of exponential attractors makes them more stable under perturbations and numerical approximations compared to global attractors. In order to derive the finite fractal dimension and then the exponential attractor, it requires the systems must satisfy a strong squeezing property, i.e. both the contractive and smooth components should be controlled by the norm of the phase space. In \cite{Chueshov2008}, Chueshov and Lasiecka generalized the squeezing properties given in \cite{Eden1994} and \cite{Ladyzhenskaya1982} by taking virtue of a precompact pseudometric, which further evolves into the quasi-stable theory in \cite{Chueshov2015}.
The drawback of exponential attractors with the desired properties is that they are not unique, and their construction can be influenced by the specific objectives of the study (\cite{Miranville1998}). The exponential attractors have been considered for a class of equations arising from mechanics and physics \cite{Efendiev2003,Zelik2005,Langa2010}, the existence of exponential attractors is as general as that of global attractors.



In this paper, we focus on the stability of non-autonomous dynamical systems under perturbations. The non-autonomous dynamical systems strictly depend on the initial moment and current moment of the system.

Generally speaking, exponential attractors possess better stability than global attractors. In most cases, global attractors can only be shown to be upper semi-continuous under perturbations (\cite{Hale11988}). However, for certain classes of dynamical systems, such as gradient systems, the global attractor may also exhibit lower semi-continuity with respect to perturbations (\cite{Hale1988}). Due to the non-uniqueness of exponential attractors, proving their stability requires careful attention to the specifics of their construction. Furthermore, the study of exponential attractors in non-autonomous dynamical systems presents additional challenges: their existence and structure become more complicated, particularly since exponential attractors are only positively invariant (\cite{Efendiev2003}).

In \cite{Zelik2005}, Efendiev, Zelik, and Miranville proposed a construction of exponential attractors for autonomous dynamical systems based on smoothness properties, which was later extended to non-autonomous systems. Within this framework, they also established the stability of exponential attractors under perturbations.
In \cite{Langa2010}, the conditions from \cite{Zelik2005} were slightly relaxed, leading to a constructive proof of the existence of exponential attractors for non-autonomous dynamical systems. 
However, the absorbing set used in this construction is time-independent, which limits its broader applicability to non-autonomous systems. 
In such systems, it is more natural to consider families of time-dependent absorbing sets rather than a single fixed set.
 Carvalho and Sonner generalized the results to pullback asymptotically compact process in \cite{Carvalho2013a}. They considered a family of time-dependent absorbing sets, which can even grow in the past instead of a fixed bounded absorbing set.
{ However, most existing studies on this topic have primarily focused on the exponential attraction property within bounded domains. In the present work, we attempt to explore a broader context by proposing a new construction of open coverings. It is observed that the set of centers of these covering balls also appears to exhibit exponential pullback attraction. This construction may offer some insights into the geometric characteristics of solutions to dissipative equations, as well as the role of dissipativity in preserving certain invariance properties. }



{In this paper, we advance the existing theory by introducing a novel approach to constructing exponential attractors for non-autonomous dynamical systems. By relaxing the classical assumptions presented in \cite{Zelik2005}, we develop a construction framework based on time-dependent absorbing sets—a generalization that offers broader applicability. A key contribution of our method is the establishment of stability for these exponential attractors under external perturbations, a result not generally available in earlier works. Specifically, we prove the existence of uniform exponential attractors for a class of perturbed systems satisfying mild structural conditions, thereby extending the robustness and utility of exponential attractors in non-autonomous settings. }
As an application of the conclusions in this paper, we study the stability of the dynamical behavior for the subquintic weakly damped  hyperbolic equation defined on a bounded domain in $\mathbb{R}^3$ under perturbations.
The weakly damped  hyperbolic equation is one of the most important models in the theory of infinite-dimensional dynamical systems and is often regarded as a benchmark for validating fundamental theoretical frameworks.  
Additionally, it serves as a mathematical model for a variety of physical phenomena, including descriptions based on the Maxwell–Cattaneo law, relativistic quantum mechanics, and electric conduction in Josephson junctions (see, e.g., \cite{Gallay1998}).

 

 The global well-posedness of weak solutions for the weakly damped  hyperbolic equation on a bounded domain in $\mathbb{R}^3$ is known only for nonlinearities $f$ with sub-cubic or cubic growth rates.
 For the weakly damped  hyperbolic equation with sub-quintic growth nonlinearity,  the equation does not satisfy the energy equality, making it impossible to establish the uniqueness of the weak solution. 
Building upon recent advances in Strichartz estimates \cite{Blair2009}, Kalantarov, Savostianov, and Zelik \cite{Kalantarov2016} extended the analysis to the quintic case. They introduced a class of solutions endowed with enhanced space-time regularity, termed Shatah–Struwe solutions (see the definition in \cite{Kalantarov2016}).
 In this paper, we study the long-term behavior of the non-autonomous dynamical system generated by the weakly damped  hyperbolic equation with sub-quintic nonlinearity, in the sense of the Shatah-Struwe solution. First, we construct exponential attractors using the approach proposed in this paper, and then we prove the stability of these exponential attractors with respect to the external forcing term.

The following is an outline of this paper. In Section 2, we recall the concepts of a non-autonomous process and pullback exponential attractors. An abstract theorem on the existence and Hölder continuity of pullback exponential attractors is also established in this section. Finally, as an application, we establish the robustness of exponential attractors for the non-autonomous damped  hyperbolic equation with sub-quintic nonlinearity and obtain the corresponding estimate for the symmetric metric.

\section{The abstract setting for non-autonomous processes}
\noindent

\subsection{Preliminaries}

We start with some preliminary notions which will used in this paper.

  \begin{definition}[\cite{CarvalhoLanga2013}]
    Let $X$ be a phase space, $\mathcal{L}(X)$ be the set of all continuous transformations from $X$ into itself. A process in $X$ is a family of maps $\{S(t,s):t>s\}$ in $\mathcal{L}(X)$ such that properties
    \begin{enumerate}
      \item $S(t,t)=I,$ for all $t\in\mathbb{R},$
      \item $S(t,s)=S(t,\tau)S(\tau,s),$ for all $t\geq\tau\geq s,$
      \item $(t,s,x)\mapsto S(t,s)x$ is continuous, $t\geq s,$ $x\in X.$
    \end{enumerate}
  \end{definition}

We give the dentition of attraction below.
\begin{definition}[Hausdorff semidistance \cite{Temam1998}]
    Let $X$ be a Banach space, and $A$ and $B$ are some some subset of the Banach space $X$. Denote by $dist_H(A,B)$ the Haudorff semidistance between $A$ and $B$, defined as
    \begin{equation}
        dist_H(A,B)=\sup_{a\in A}\inf_{b\in B}d(a,b)=\sup_{a\in A}\inf_{b\in B}\|a-b\|_X.
    \end{equation}
\end{definition}
Note that $dist_H(A,B)=0$ implies only that $\overline{A}\subseteq\overline{B}$, where $\overline{M}$ denotes the closure of $M$ in $X$; we only have $dist_H(A,B)=0$ implying $A\subset B$ provided that $B$ is closed. In this paper we will also use the symmetric Hausdorff metric:
\begin{equation*}
dist^{symm}_H(A,B)=\max(dist_H(A,B),dist_H(B,A)).
\end{equation*}

\begin{definition}[Pullback attraction \cite{Temam1998}]
    Let $B$ and $C$ be subsets of $X$. We say that $B$ pullback attracts $C$ under a process $\{S(t,s):t>s\}$ if 
    \begin{equation*}
       \lim_{t-s\to\infty} dist_H(S(t,s)C,B) = 0.
    \end{equation*}
\end{definition}

In the theory of infinite dynamical systems, we often consider the fractal dimension as an index for characterizing sets by quantifying their complexity as a ratio of the change in detail to the change in scale.

\begin{definition}[Fractal dimension \cite{Temam1998}]
    Let $X$ be a Banach space and $B$ be a subset of $X$. The fractal dimension of $B$, which is also called the capacity of $B$, is the number 
    \begin{equation*}
        d_F(B) = \limsup_{\epsilon\to 0}\frac{\log N_B(\epsilon)}{\log 1/\epsilon},
    \end{equation*}
    where $ N_B(\epsilon)$, $\epsilon>0$, denote the minimal number of balls of $X$ of radius $\epsilon$ which is necessary to cover $B$.
\end{definition}

Exponential attractors are mathematical constructs that play a key role in the analysis of dynamical systems, especially those of infinite dimension arising from partial differential equations. They are characterized by several important properties, such as finite dimensionality and the exponential rate at which trajectories converge toward them.

  \begin{definition}[Pullback exponential attractors \cite{Carvalho2013a}]
    Let $\{S(t,s)|t\geq s\}$ be an evolution process in a metric space $X$. We call the family $\mathcal{M}=\{\mathcal{M}(t)|t\in\mathbb{R}\}$ a pullback exponential attractor for evolution process $\{S(t,s)|t\geq s\}$ in $X$ if
    \begin{enumerate}
      \item the subsets $\mathcal{M}(t)\subset X$ are non-empty and compact in $X$ for all $t\in\mathbb{R}$;
      \item the family is positive semi-invariant, that is
            \begin{equation*}
              \mathcal{M}(s)\subset\mathcal{M}(t)~ \text{for all}~~t\geq s;
            \end{equation*}
      \item the fractal dimension in $X$ of the section $\mathcal{M}(t), t\in\mathbb{R},$ is uniformly bounded;
      \item the family $\{\mathcal{M}(t)|t\in\mathbb{R}\}$; exponentially pullback attracts bounded subsets of $X$; that is, there exists a positive constant $\omega>0$ such that for every bounded subset $D\subset X$ and $t\in\mathbb{R}$
          \begin{equation}
            \lim_{s\to\infty} e^{\omega s}\mathrm{dist}_H(S(t,t-s)D,\mathcal{M}(t))=0.
          \end{equation}
    \end{enumerate}
  \end{definition}

 Exponential attractors are stable under small perturbations. This property is useful in applications where the system may be subject to external influences or noise, as the attractor will still accurately represent the long-term behavior of the system despite these perturbations. This is the primary reason for studying the stability of exponential attractors in this paper.

\subsection{Constructive methods for exponential attractors in non-autonomous dynamical system}
\noindent

Exponential attractors provide a powerful tool for understanding the long-term dynamics of complex systems, especially in cases where global attractors are not easily obtainable or are infinite-dimensional.
 Exponential attractors can often be constructed explicitly using certain mathematical techniques. This allows researchers to directly compute or approximate the attractor for specific systems, especially to prove the stability of the dynamical system under perturbation.

In this section, we construct a family of exponential attractors for an abstract operator generated by the  non-autonomous dynamical system. This non-autonomous operator, often referred to as a dynamical process $U(t,\tau)$, depends on two parameters $t,\tau \in \mathbb{R}$ (or $t,\tau \in \mathbb{Z}$ for discrete times), with $t \geq \tau$. When the operator is Hölder continuous with respect to the perturbation, it can also be shown that the exponential attractor is Hölder continuous with respect to the perturbation.

The admissible class of processes considered in this paper is defined as follows:

\begin{definition}\label{def1}
Let $H_1$ and $H$ be two Banach spaces such that $H_1$ is compactly embedded into $H$ and let $B$ be a bounded subset of $H_1$. An operator $U:H_1\to H_1$ belongs to the class $\mathbb{U}_{\mu,K}(B)$ for some positive constants $\mu$, $K$ and two operator $S$ and $C$ such that if for every point $h_1,h_2\in B$, we have
\begin{align}
       &U=S+C, \nonumber \\
   &\|Sh_1-Sh_2\|_{H_1}  \leq  \frac{1-\mu}{2}\|h_1-h_2\|_{H_1},\nonumber \\
   &\|Ch_1-Ch_2\|_{H_1}  \leq  K \|h_1-h_2\|_{H}.
\end{align}\label{ass1}
\end{definition}

A dynamical process $U$ in the phase space $\Phi$ is a two-parametrical family of maps $\{U(l,m):~ \Phi\to\Phi,~ l,m\in\mathbb{Z},~l\geq m\}$ such that:
\begin{equation*}
  U(l,k)\cdot U(k,m) = U(l,m),~U(m,m)=Id, l,m,k\in\mathbb{Z}, l\geq k\geq m.
\end{equation*}
We set $U(n):=U(n+1,n)$.

The main result of this section is the following theorem, which gives a H\"{o}lder continuous family of exponential attractors associated with the maps $U\in\mathbb{U}_{\mu,K}(B)$.

\begin{theorem}\label{th1}
Let the spaces $H_1$ and $H$ be defined as in Definition \ref{def1}. Assume that there exist bounded sets $\{B(n)\}_{n\in\mathbb{Z}}$ in $H_1$ and every dynamical process $U$ on $H_1$ with $U(n)\in\mathbb{U}_{\mu,K}(B(n))$ satisfies $U(n)B(n)\subset B(n+1)$. Moreover, we assume that the diameter of the family of bounded sets $\{B(n)\}_{n\in\mathbb{Z}}$ growths at most sub-exponentially in the past.
Then there exist pullback exponential attractors $n\to M_U(n),~n\in\mathbb{Z}$ which satisfies the following properties:\\
(1) The pullback exponential attractor $M_U(n)\subset B(n)$ for every $n\in\mathbb{Z}$ and its fractal dimension is finite, i.e.,
\begin{equation}\label{con1}
  \dim_F(M_U(n),H_1)\leq C,
\end{equation}
where the constant $C$ depends only on $H,~H_1,~\mu,~ K$ and is independent of $n$.\\
(2) The family $\{M_U(n)\}$ is semi-invariant with respect to $U$, i.e.,
\begin{equation}\label{con2}
  U(k,m)M_U(m)\subset M_U(k),
\end{equation}
for all $k,m\in\mathbb{Z}$, $k\geq m$.\\
(3) This family of sets ${M_U(n)}{n\in\mathbb{Z}}$ enjoys a uniform exponential attraction property of the form
\begin{equation}\label{con3}
  \mathrm{dist}_{H_1}(U(n+k,n)B(n),M_U(n+k)) \leq C e^{-\alpha k},
\end{equation}
where the constants $C$ and $\alpha$ depend only on $H,~H_1,~\mu,~ K$ and are independent of $n\in\mathbb{Z}$ and $k\in\mathbb{N}$.\\
(4) The map $U \to M_U(n)$ is uniformly Hölder continuous in the following sense: for any two dynamical processes $U_1$ and $U_2$ satisfying $U_1(n), U_2(n) \in \mathbb{S}{\mu,K}(B(n))$, we have
\begin{equation}\label{con4}
  \mathrm{dist}_{H_1}^{\mathrm{symm}}(M_{U_1}(n),M_{U_2}(n))\leq C \sup_{l\in(-\infty,n)}\{e^{-\beta(n-l)}\|U_1(n-l)-U_2(n-l)\|_{\mathbb{U}}^\vartheta\},
\end{equation}
where the positive constants $C$, $\beta$, and $\vartheta$ depend only on $H$, $H_1$, $\mu$, and $K$, and are independent of $n$ and the particular choices of $U_1$ and $U_2$. Here,
$$\|U_1(n-l)-U_2(n-l)\|_{\mathbb{U}}:= \sup_{u\in B(n-l)}\|U_1(n-l)u-U_2(n-l)u\|_{H_1}.$$
\end{theorem}
The constant $C$ in the following proof may vary from line to line.

\begin{proof}
For any $n\in\mathbb{Z}$  and any $k\leq n$, there exist some constant $R_{n-k}>0$ and $v_{n-k}\in B(n-k)$ such that the bounded set $B(n-k)\subset B_{H_1}(v_{n-k},R_{n-k})$, where $B_{H_1}(v_{n-k},R_{n-k})$ denote a ball in $H_1$ centered in $v_{n-k}$ with radius $R_{n-k}$. For $k=1$, since $H_1$ is embedded in $H$ compactly, then there exists a finite set $V_0(n-1):=\{h_1,\cdots,h_{N_0}\}\subset B(n-1)$ such that
\begin{equation}\label{thp1}
  B(n-1)\subset \bigcup_{i=1}^{N_0}B_H(h_i,R_{n-1}\frac{\mu}{2K}).
\end{equation}

Then for every $y\in U(n,n-1)B(n-1)$ and some $x\in B(n-1)$ we have $y=U(n,n-1)x$, and for this specific $x$, according to \eqref{thp1} there exists $h_i\in V_0(n-1)$ such that $\|x-h_i\|_H\leq R_{n-1}\frac{\mu}{2K}$, then follows the estimate below:
\begin{align*}
  \|y-U(n,n-1)h_i\|_{H_1}   =& \|U(n,n-1)x-U(n,n-1)h_i\|_{H_1}\\
    & \leq \frac{1-\mu}{2}\|x-h_i\|_{H_1} + K\|x-h_i\|_H\\
    & \leq (1-\frac{\mu}{2})R_{n-1}.
\end{align*}
Let $V_1(n):=U(n,n-1)V_0(n-1)$ and $r_1(n):=(1-\frac{\mu}{2})R_{n-1}$, we can conclude that $U(n,n-1)B(n-1)$ can be covered by finite balls in $H_1$ centered in $V_1(n)$ with radius $r_1(n)$.

For any $n\in\mathbb{Z}$  and any $k\in\mathbb{N}$, we want to establish a finite covering for $U(n,n-k)B(n-k)$ in $H_1$ such that $U(n,n-k)B(n-k)$ is contained in balls of radius $r_k(n):=R_{n-k}(1-\frac{\mu}{2})^k$ with centers in $V_k(n)$. Note that the case of $k=1$ has been proved. We assume that the covering has been constructed for $l\leq n$, i.e. $U(n,n-l)B(n-l)$ can be covered by balls of radius $r_l(n):=R_{n-l}(1-\frac{\mu}{2})^l$ with centers in $V_l(n)$.

For $k=l+1$, we have
\begin{equation}\label{pf1}
  U(n,n-(l+1))B(n-(l+1)) = U(n,n-1)U(n-1,n-1-l)B(n-1-l).
\end{equation}
Due to our assumptions, we can obtain
\begin{align*}
  U(n-1,n-1-l) B(n-1-l)  &\subset \bigcup_{h\in V_l(n-1)}B_{H_1}(h,r_l(n-1)) \\
    & =\bigcup_{h\in V_l(n-1)}B_{H_1}(h,R_{n-1-l}(1-\frac{\mu}{2})^l).
\end{align*}

Consider the unit ball $B_{H_1}(0,1)$, there exists a finite covering in $H$ with radius $\frac{\mu}{4K}$ since the embedding $H_1\hookrightarrow H$ is compact. Let $\mathbb{V}:=\{u_1,\cdots,u_i,\cdots,u_N\}$ be the centers of the covering balls in $H$. For any $h\in V_l(n-1)$, we can construct a covering for every ball $B_{H_1}(h, R_{n-1-l}(1-\frac{\mu}{2})^l)$ in $H$, i.e.,
\begin{equation}\label{pf2}
  B_{H_1}(h, R_{n-1-l}(1-\frac{\mu}{2})^l)\subset \bigcup_{i=1}^N B_H(h+R_{n-1-l}(1-\frac{\mu}{2})^lu_i,R_{n-1-l}(1-\frac{\mu}{2})^l\frac{\mu}{4K}),
\end{equation}
for some $1\leq i\leq N$.

For any $u,w\in U(n-1,n-1-l)B(n-1-l)\cap B_{H_1}(h,r_l(n-1))\cap B_H(h+r_l(n-1)u_i,r_l(n-1)\frac{\mu}{4K})$, we have
\begin{align}\label{pf3}
  \|C(n,n-1)u-C(n,n-1)w\|_{H_1} & \leq K\|u-w\|_H \nonumber \\
   & \leq K\cdot 2R_{n-1-l}(1-\frac{\mu}{2})^l\frac{\mu}{4K} \nonumber \\
    & = \frac{\mu}{2}R_{n-1-l}(1-\frac{\mu}{2})^l,
\end{align}
then there exists $w_i^h\in U(n,n-1-l)B(n-1-l)$ such that:
\begin{align}\label{pf4}
  C(n,n-1)\big(U(n-1,n-1-l)B(n-1-l)\cap B_{H_1}(h, R_{n-1-l}(1-\frac{\mu}{2})^l)\big)  \nonumber \\
    \subset \bigcup_{i=1}^N B_{H_1}(C(n,n-1)w_i^h,\frac{\mu}{2}R_{n-1-l}(1-\frac{\mu}{2})^l).
\end{align}

For any $u\in U(n-1,n-1-l)B(n-1-l)\cap B_{H_1}(h,R_{n-1-l}(1-\frac{\mu}{2})^l)$ and $1\leq i \leq N$, we have
\begin{align}\label{pf5}
  \|S(n,n-1)u-S(n,n-1)y^h_i\|_{H_1}  & \leq \frac{1-\mu}{2}\|u-y^h_i\|_{H_1} \nonumber \\
    & \leq (1-\mu)R_{n-1-l}(1-\frac{\mu}{2})^l.
\end{align}

Combine \eqref{pf4} and \eqref{pf5}, since $U(n,n-1)=S(n,n-1)+C(n,n-1)$ we have
\begin{align}\label{pf66}
  & U(n,n-1)\left(U(n-1,n-1-l)B(n-1-l)\cap B_{H_1}(h,R_{n-1-l}(1-\frac{\mu}{2})^l)\right) \nonumber  \\
  & = (S(n,n-1)+C(n,n-1))\left(U(n-1,n-1-l)B(n-1-l)\cap B_{H_1}(h,R_{n-1-l}(1-\frac{\mu}{2})^l)\right) \nonumber \\
  & \subset \bigcup_{i=1}^N B_{H_1}(U(n,n-1)y_i^h, R_{n-1-l}(1-\frac{\mu}{2})^{l+1}).
\end{align}

Let $V_{l+1}(n):= \bigcup_{h\in V_l(n-1)}\bigcup_{i=1}^NU(n,n-1)y_i^h$ and $r_{l+1}(n):=R_{n-1-l}(1-\frac{\mu}{2})^{l+1}$, we can conclude that $U(n,n-1-l)B(n-1-l)$ can be covered by balls centered in $V_{l+1}(n)$ with radius $r_{l+1}(n)$. Then for every $n\in\mathbb{Z}$ and every $k\in\mathbb{N}$, we have established a finite covering for $U(n,n-k)B(n-k)$ in $H_1$ such that $U(n,n-k)B(n-k)$ is contained in balls of radius $r_k(n):=R_{n-k}(1-\frac{\mu}{2})^k$ with centers in $V_k(n)$.

Let $\sharp$ denote the cardinality of a set.
According to the construction, we can conclude that
\begin{align}
 & \sharp V_k(n) = N_0N^{k-1},\label{pf6}\\
  & V_k(n)\subset U(n,n-k)B(n-k) \label{pf7}\\
  & \mathrm{dist}_{H_1}(U(n,n-k)B(n-k),V_k(n))\leq R_{n-k}(1-\frac{\mu}{2})^k,\label{pf8}\\
  & \mathrm{dist}^{\mathrm{symm}}_{H_1}(V_{k+1}(n), U(n,n-1)V_k(n-1))\leq R_{n-1-k}(1-\frac{\mu}{2})^{k+1} \label{pf9}.
\end{align}

Note that the conclusion \eqref{pf6}-\eqref{pf8} can be proved directly from the construction, we prove \eqref{pf9} in the following.

Noticed that $V_{k-1}(n-1)\subset U(n-1,n-1-k)B(n-1-k)$ we have
\begin{align}\label{pf10}
  & \mathrm{dist}_{H_1}(U(n,n-1)V_k(n-1),V_{k+1}(n)) \nonumber   \\
  & \leq\mathrm{dist}_{H_1}(U(n,n-1)U(n-1,n-1-k)B(n-1-k),V_{k+1}(n)) \nonumber \\
  & =\mathrm{dist}_{H_1}(U(n,n-1-k)B(n-1-k),V_{k+1}(n)) \nonumber \\
  & \leq  R_{n-1-k}(1-\frac{\mu}{2})^{k+1}.
\end{align}

On the other hand, we can see the fact that $V_{k+1}(n)=\bigcup_{h\in V_k(n-1)}U(n,n-1)V_{k+1}(n-1,h),$ where
\begin{equation}\label{vnh}
  V_{k+1}(n-1,h):= \bigcup_{i=1}^Ny_i^h.
\end{equation}

By the definition of Hausdorff distance, we have
\begin{align*}
  & \mathrm{dist}_{H_1}(V_{k+1}(n),U(n,n-1)V_k(n-1))\nonumber\\
    & = \max_{h\in V_k(n-1)}\mathrm{dist}_{H_1}(U(n,n-1)V_{k+1}(n-1,h),U(n,n-1)V_k(n-1)) \nonumber \\
   &  = \mathrm{dist}_{H_1}(U(n,n-1)V_{k+1}(n-1,\hat{h}),U(n,n-1)V_k(n-1)) \nonumber \\
    & \leq \mathrm{dist}_{H_1}(U(n,n-1)V_{k+1}(n-1,\hat{h}),U(n,n-1)\hat{h}),
\end{align*}
for some $\hat{h}\in V_k(n-1)$.

Due to the construction of $V_k(n)$, for some $1\leq i_0 \leq N$ we can deduce that :
\begin{align}\label{pf12}
  & \mathrm{dist}_{H_1}(V_{k+1}(n),U(n,n-1)V_k(n-1))\nonumber\\
    & \leq \mathrm{dist}_{H_1}(U(n,n-1)y_{i_0}^{\hat{h}},U(n,n-1)\hat{h})\nonumber \\
    & \leq \frac{1-\mu}{2}\|y_{i_0}^{\hat{h}}-\hat{h}\|_{H_1} + K \|y_{i_0}^{\hat{h}}-\hat{h}\|_H \nonumber\\
    & \leq \frac{1-\mu}{2}\cdot 2R_{n-1-k}(1-\frac{\mu}{2})^k + K\cdot2R_{n-1-k}(1-\frac{\mu}{2})^k\frac{\mu}{4K} \nonumber\\
    & = R_{n-1-k}(1-\frac{\mu}{2})^{k+1}.
\end{align}
Thus the validity of \eqref{pf9} follows from the inequality \eqref{pf10} and \eqref{pf12}.

We define the sets $E_k(n)$ by
\begin{equation}\label{pfe}
  E_1(n)=V_1(n), \quad E_k(n)=V_k(n)\cup U(n,n-1)E_{k-1}(n-1).
\end{equation}
From the construction, we can get the following conclusions:
\begin{align}
  & \sharp E_k(n) = kN_0N^k,\label{pf13}\\
  & U(n+1,n)E_k(n)\subset E_{k+1}(n+1), \label{pf14}\\
  & \mathrm{dist}_{H_1}(U(n,n-k)B(n-k),E_k(n))\leq R_{n-k}(1-\frac{\mu}{2})^k.\label{pf15}
\end{align}
Properties \eqref{pf13}-\eqref{pf15} can be simply deduced from the properties of $V_k(n)$ and the construction of $E_k(n)$.

Then we can construct the required exponential attractors $\{M(n)\}_{n\in\mathbb{Z}}$ in the following way:
\begin{equation}\label{pf16}
  M^\prime(n):=\bigcup_{k=1}^\infty E_k(n) \qquad \text{and}\qquad M(n)=\big[M^\prime(n)\big]_{H_1},
\end{equation}
where $[ \cdot ]_{H_1}$ denote the closure in $H_1$.

In order to show that $\{M(n)\}_{n\in\mathbb{Z}}$ is the desired family of exponential attractors in Theorem \ref{th1}, we need to verify that $M(n)$ satisfies the properties of semi-invariance, exponential attraction, finite fractal dimension and H\"{o}lder continuity. In fact, the semi-invariance \eqref{con2} and the exponential attraction property \eqref{con3}  are straightforward consequences of \eqref{pf14}, \eqref{pf15} and \eqref{pf16}. The proof of the latter two properties requires two conclusions as auxiliary.
\end{proof}

\begin{lemma}\label{fuzhu1}Let $E_k(n)$ be as in \eqref{pfe}. Then there exist constants $C,~\alpha$ depending on $\mu,~K$ such that
\begin{equation}\label{fpf}
  \mathrm{dist}_{H_1}(E_k(n),U(n,n-l)B(n-l))\leq CR_{n-k}(1-\frac{\mu}{2})^{\alpha k},
\end{equation}
for all $n\in\mathbb{Z}$, $k\in\mathbb{N}$ and $l\leq\alpha k$.
\end{lemma}

\begin{proof}
First we estimate $\mathrm{dist}_{H_1}(U(n,n-m)V_{k-m}(n-m),U(n,n-l)V_{k-l}(n-l))$ for all $n\in\mathbb{Z}$, $k\in\mathbb{N}$ and $0\leq m <l$.

By the triangle inequality we have:
\begin{align}\label{fpf1}
   & \mathrm{dist}_{H_1}(U(n,n-m)V_{k-m}(n-m),U(n,n-l)V_{k-l}(n-l))\nonumber  \\
   &  \leq \mathrm{dist}_{H_1}(U(n,n-m)V_{k-m}(n-m),U(n,n-m-1)V_{k-m-1}(n-m-1))+ \nonumber\\
    &  \cdots +\mathrm{dist}_{H_1}(U(n,n-m-i)V_{k-m-i}(n-m-i),U(n,n-m-i-1)V_{k-m-i-1}(n-m-i-1))\nonumber\\
    &   +\cdots + \mathrm{dist}_{H_1}(U(n,n-l+1)V_{k-l+1}(n-l+1),U(n,n-l)V_{k-l}(n-l),
\end{align}
where $0\leq i \leq l-m-1$.

For all $n\in\mathbb{Z}$, $j\in\mathbb{N}$ and any $A_1,A_2\subset B(n-j)$, according to assumption \eqref{ass1}, the following inequality can be obtained by induction:
\begin{equation}\label{fpf2}
  \begin{aligned}
    & \mathrm{dist}^{\mathrm{symm}}_{H_1}(U(n,n-j)A_1,U(n,n-j)A_2)\\
   &  \leq  (\frac{1-\mu}{2})^j\mathrm{dist}^{\mathrm{symm}}_{H_1}(A_1,A_2) + K^j\mathrm{dist}^{\mathrm{symm}}_H(A_1,A_2).
  \end{aligned}
\end{equation}

Together with \eqref{fpf2} and property \eqref{pf9} we have:
\begin{align}\label{fpf3}
    & \mathrm{dist}_{H_1}(U(n,n-j)V_{k-j}(n-j),U(n,n-j-1)V_{k-j-1}(n-j-1))\nonumber \\
  &  \leq  (\frac{1-\mu}{2})^jR_{n-k}(1-\frac{\mu}{2})^{k-j} + K^jR_{n-k}(1-\frac{\mu}{2})^{k-j}.
\end{align}

Then substitute \eqref{fpf3} into \eqref{fpf1} we obtain:
\begin{align}\label{fpf33}
    & \mathrm{dist}_{H_1}(U(n,n-m)V_{k-m}(n-m),U(n,n-l)V_{k-l}(n-l))\nonumber \\
   &  \leq \sum_{i=1}^l\left((\frac{1-\mu}{2})^{l-i}R_{n-k}(1-\frac{\mu}{2})^{k+i-l} + K^{l-i}R_{n-k}(1-\frac{\mu}{2})^{k+i-l} \right) \nonumber\\
   &  \leq R_{n-k}\left(\frac{(\frac{1-\mu}{2})^{l-1}(1-\frac{\mu}{2})^{k+1-l}-(1-\frac{\mu}{2})^k}{1-(\frac{1-\nu}{2})^{-1}(1-\frac{\mu}{2})} + \frac{K^{l-1}(1-\frac{\mu}{2})^{k+1-l}-(1-\frac{\mu}{2})^k}{1-K^{-1}(1-\frac{\mu}{2})}\right)\nonumber \\
   &  \leq R_{n-k}\left(\frac{(\frac{1-\mu}{2})^{l-1}(1-\frac{\mu}{2})^{k+1-l}}{1-(\frac{1-\nu}{2})^{-1}(1-\frac{\mu}{2})} + \frac{K^{l-1}(1-\frac{\mu}{2})^{k+1-l}}{1-K^{-1}(1-\frac{\mu}{2})}\right).
\end{align}
By choosing $$\alpha = \min\left\{\frac{\log_2(1-\frac{\mu}{2})}{\log_2(\frac{1-\mu}{2})^{-1}-2\log_2(1-\frac{\mu}{2})},\frac{\log_2(1-\frac{\mu}{2})}{\log_2K-2\log_2(1-\frac{\mu}{2})}\right\},$$
then for $l\leq\alpha k$ we have:
\begin{equation}\label{fpf4}
  \mathrm{dist}_{H_1}(U(n,n-m)V_{k-m}(n-m),U(n,n-l)V_{k-l}(n-l)) \leq C R_{n-k}(1-\frac{\mu}{2})^{\alpha k},
\end{equation}
where the constants $C$ and $0<\alpha<1$ depend only on $\mu,~K$.

By the recursive definition \eqref{pfe} we can deduce that $E_k(n)=\sum_{m=0}^{k-1}U(n,n-m)V_{k-m}(n-m)$, then follows
\begin{align}\label{fpf5}
    & \mathrm{dist}_{H_1}(E_k(n),U(n,n-l)B(n-l)) \nonumber \\
    & = \mathrm{dist}_{H_1}(\sum_{m=0}^{k-1}U(n,n-m)V_{k-m}(n-m),U(n,n-l)B(n-l)).
\end{align}

For all $n\in\mathbb{Z}$, $k\in\mathbb{N}$, $0\leq m <l$ and $l\leq\alpha k$, we have
\begin{align}\label{fpf6}
    & \mathrm{dist}_{H_1}(U(n,n-m)V_{k-m}(n-m),U(n,n-l)B(n-l)) \nonumber \\
    & = \mathrm{dist}_{H_1}(U(n,n-m)V_{k-m}(n-m),U(n,n-l)V_{k-l}(n-l)) \nonumber\\
   &  + \mathrm{dist}_{H_1}(U(n,n-l)V_{k-l}(n-l),U(n,n-l)B(n-l)) \nonumber\\
  &   \leq C R_{n-k}(1-\frac{\mu}{2})^{\alpha k},
\end{align}
where we used the inequality \eqref{fpf4} and the property $V_{k-l}(n-l)\subset U(n-l,n-k)B(n-k)$ by \eqref{pf7}.

In addition, for all $n\in\mathbb{Z}$, $k\in\mathbb{N}$, $m\geq l$ and $l\leq\alpha k$, due to the property \eqref{pf7} and the semi-invariance of $\{B(s)\}_{s\in\mathbb{Z}}$ we have:
\begin{align*}
 & U(n,n-m)V_{k-m}(n-m)   \subset U(n,n-m)U(n-m,n-k)B(n-k) \\
    & \subset U(n,n-k)B(n-k)= U(n,n-l)U(n-l,n-k)B(n-k)\\
   &  \subset U(n,n-l)B(n-l).
\end{align*}
Thus the following formula holds:
\begin{equation}\label{pfp7}
  \mathrm{dist}_{H_1}(U(n,n-m)V_{k-m}(n-m),U(n,n-l)B(n-l)) = 0.
\end{equation}

Finally, the lemma is proved by using inequalities \eqref{fpf5}, \eqref{fpf6} and \eqref{pfp7}.
\end{proof}

Let us verify the finite dimensionality for $\{M(n)\}_{n\in\mathbb{Z}}$. Let $\varepsilon>0$ be sufficiently small, we need to construct a covering of $M(n)$ by a finite number of $\varepsilon-$balls in $H_1$. By applying the property \eqref{pf8}, we can cover $U(n,n-m)B(n-m)$ with $\frac{\varepsilon}{2}-$balls in $H_1$ centered at $V_{m}(n)$ where $m$ satisfies $$R_{n-m}(1-\frac{\mu}{2})=\frac{\varepsilon}{2}\quad\text{and}\quad m=m(\varepsilon)=(1-\frac{\mu}{2})^{-1}\log_2\frac{2R_{n-m}}{\varepsilon}.$$

In addition, by using Lemma \ref{fuzhu1} we can cover $E_k(n)$ with $\varepsilon$-balls in $H_1$ centered at $V_{m}(n),$ where $k$ satisfies:
\begin{align*}
    & \mathrm{dist}_{H_1}(E_k(n),V_m(n)) \\
   &  \leq \mathrm{dist}_{H_1}(E_k(n),U(n,n-m)B(n-m)) + \mathrm{dist}_{H_1}(U(n,n-m)B(n-m),V_m(n))\\
   &  \leq CR_{n-m}(1-\frac{\mu}{2})^m + \frac{\varepsilon}{2} \leq \varepsilon.
\end{align*}
Then the assertion in below follows:
\begin{equation*}
  k\geq k(\varepsilon) = (1-\frac{\mu}{2})^{-1}\log_2\frac{2CR_{n-m}}{\varepsilon}.
\end{equation*}

Note that $$M(n)=\left[M^{\prime}(n)\right]_{H_1}=\left[\bigcup_{k=1}^\infty E_k(n)\right]_{H_1},$$
combining with the discussion above we have:
\begin{align*}
  N_\varepsilon(M^\prime(n),H_1)  & \leq \sharp V_m(n) + \sum_{k\leq k(\varepsilon)} E_k(n) \\
    & \leq N_0N^{m(\varepsilon)-1} + N_0\sum_{k\leq k(\varepsilon)}kN^k\\
   &  \leq N_0(1+m^2(\varepsilon))N^{k(\varepsilon)},
\end{align*}
then the fractional dimension of $M(n)$ can be estimated as:
\begin{align}\label{dpf}
  \dim_F(M(n),H_1)  & = \dim_F(M^\prime(n),H_1)\nonumber\\
    & =\limsup_{\varepsilon\to0}\frac{\log_2 N_\varepsilon(M(n),H_1)}{\log_2 \frac{1}{\varepsilon}}\nonumber \\
     &\leq (1-\frac{\mu}{2})^{-1}\log_2N,
\end{align}
where we used the assumption that the radius of $\{B(s)\}_{s\in\mathbb{Z}}$ is sub-exponential growth in the past. We can also deduce from \eqref{dpf} that $M(n)$ is precompact for all $n\in\mathbb{Z}$.

Next we need one more auxiliary lemma to prove the H\"{o}lder continuity for $M^U(n)$.
\begin{lemma}\label{fuzhu2}
For any $n\in\mathbb{Z}$ and $k\in\mathbb{N}$, let $\{B(n)\}_{n\in\mathbb{Z}}$ satisfies the assumptions in Theorem \ref{th1} and $\{E_k^U(n)\}_{n\in\mathbb{Z}}$ be constructed as in \eqref{pfe} for every $U(n)\in\mathbb{U}_{\mu,K}(B(n))$. Then for any $U_1(n),U_2(n)\in\mathbb{U}_{\mu,K}(B(n))$, the following inequality holds:
\begin{equation}\label{hpf}
  \mathrm{dist}_{H_1}^{\mathrm{symm}}\left(E_k^{U_1}(n),E_k^{U_2}(n)\right) \leq C\sup_{l\in(0,k]}\|U_1(n-l)-U_2(n-l)\|_{\mathbb{U}},
\end{equation}
where the constant $C$ depends only on $\mu,K$.
\end{lemma}

\begin{proof}We estimate the Hausdorff distance between $V_k^{U_1}(n)$ and  $V_k^{U_2}(n)$ first. Define $$d_k(n)=\mathrm{dist}_{H_1}(V_k^{U_1}(n),V_k^{U_2}(n)).$$
For any $n\in\mathbb{Z}$ and $k\in\mathbb{N}$, our aim is to prove that
\begin{equation}\label{hpf1}
  d_k(n) \leq C\sup_{l\in(0,k]}\|U_1(n-l)-U_2(n-l)\|_{\mathbb{U}},
\end{equation}
where the constant $C$ depends only on $\mu,K$.
Note that $V_0(n-1)$ is same for every $U(n)\in\mathbb{U}_{\mu,K}(B(n))$, then
\begin{align*}
  d_1(n)  & = \mathrm{dist}_{H_1}(V_1^{U_1}(n),V_1^{U_2}(n)) \\
    &  = \mathrm{dist}_{H_1}(U_1(n-1)V_0(n-1),U_2(n-1)V_0(n-1))\\
    & \leq \|U_1(n-l)-U_2(n-l)\|_{\mathbb{U}},
\end{align*}
which implies that the inequality \eqref{hpf1} holds for $k=1$.

We assume that for any $n\in\mathbb{Z}$ the inequality \eqref{hpf1} holds for $k=m$. Then for any $h_1\in V_m^{U_1}(n-1)$, there exists $h_2\in V_m^{U_2}(n-1)$ such that
\begin{equation*}
  \|h_1-h_2\|_{H_1} \leq d_m(n-1).
\end{equation*}
According to the construction of $V_k(n-1,h)$ in \eqref{vnh}, we conclude that for any $\bar{h}_1\in V^{U_1}_k(n-1,h_1)$, there exists $\bar{h}_2\in V^{U_2}_k(n-1,h_2)$ such that
\begin{equation*}
  \|\bar{h}_1-\bar{h}_2\|_{H_1} \leq d_m(n-1).
\end{equation*}

Then we can obtain the following estimate by using \eqref{ass1}:
\begin{align}\label{hpf2}
    & \|U_1(n,n-1)\bar{h}_1-U_2(n,n-1)\bar{h}_2\|_{H_1} \nonumber\\
    & \leq \|U_1(n,n-1)\bar{h}_1-U_1(n,n-1)\bar{h}_2\|_{H_1} + \|U_1(n,n-1)\bar{h}_2-U_2(n,n-1)\bar{h}_2\|_{H_1}\nonumber\\
   &  \leq\frac{1-\mu}{2}\|\bar{h}_1-\bar{h}_2\|_{H_1} + K \|\bar{h}_1-\bar{h}_2\|_H + \|U_1(n-1)-U_2(n-1)\|_{\mathbb{U}}\nonumber\\
   &  \leq C\left(\frac{1-\mu}{2} + K \right) \sup_{l\in(0,m]}\|U_1(n-1-l)-U_2(n-1-l)\|_{\mathbb{U}} + \|U_1(n-1)-U_2(n-1)\|_{\mathbb{U}} \nonumber\\
    & \leq C \sup_{l\in(0,m+1]}\|U_1(n-l)-U_2(n-l)\|_{\mathbb{U}}.
\end{align}

For $k=m+1$, by the definition of the Hausdorff distance we have:
\begin{align}\label{hpf3}
  d_{m+1}(n) &  = \mathrm{dist}_{H_1}(V_{m+1}^{U_1}(n),V_{m+1}^{U_2}(n))\nonumber\\
   &  =\max_{u\in V_{m+1}^{U_1}(n)}\mathrm{dist}_{H_1}(u,V_{m+1}^{U_2}(n))\nonumber\\
    & = \mathrm{dist}_{H_1}(u_0,V_{m+1}^{U_2}(n))\nonumber\\
  &   = \mathrm{dist}_{H_1}(U_1(n-1)\bar{h}_1,V_{m+1}^{U_2}(n))\nonumber\\
   &  \leq \mathrm{dist}_{H_1}(U_1(n-1)\bar{h}_1,U_2(n-1)\bar{h}_2),
\end{align}
for some $u_0=U_1(n-1)\bar{h}_1$.

Then we can deduce the inequality \eqref{hpf1} for $k=m+1$ from \eqref{hpf2} and \eqref{hpf3}. Thus the assertion \eqref{hpf1} is proved by deduction.

Now we estimate the Hausdorff distance between $E_k^{U_1}(n)$ and  $E_k^{U_2}(n)$, define
\begin{equation*}
  D_k(n) :=  \mathrm{dist}_{H_1}(E_k^{U_1}(n),E_k^{U_2}(n)).
\end{equation*}
For any $n\in\mathbb{Z}$ and $k\in\mathbb{N}$, our purpose is to prove that
\begin{equation}\label{hpf4}
  D_k(n) \leq C\sup_{l\in(0,k]}\|U_1(n-l)-U_2(n-l)\|_{\mathbb{U}},
\end{equation}
where the constant $C$ depends only on $\mu,K$.
For $k=1$, we have $E_1^{U_1}(n)=V_1^{U_1}(n)$ and $E_1^{U_2}(n)=V_1^{U_2}(n)$, then
\begin{equation*}
  D_1(n) = d_1(n) \leq \|U_1(n-l)-U_2(n-l)\|_{\mathbb{U}}.
\end{equation*}

We assume that for any $n\in\mathbb{Z}$ the inequality \eqref{hpf4} holds for $k=m$.

For any $A_1,A_2\subset B(n)$, by applying the assumption \eqref{ass1} and the symmetry between $A_1$ and $A_2$ we can deduce the following inequality:
\begin{align}\label{hpf5}
  & \mathrm{dist}_{H_1}^{\mathrm{symm}}(U_1(n)A_1,U_2(n)A_2) \nonumber\\
    & \leq (\frac{1+\mu}{2}+K)\mathrm{dist}_{H_1}^{\mathrm{symm}}(A_1,A_2) + \|U_1(n)-U_2(n)\|_{\mathbb{U}}.
\end{align}
Then we estimate $D_{m+1}(n)$ by using \eqref{hpf5} and inequality \eqref{hpf4} for $k=m$. For $k=m+1$, we have:
\begin{eqnarray*}
  E_{m+1}^{U_1}(n)= V_{m+1}^{U_1}(n)\cup U_1(n-1)E_{m}^{U_1}(n-1), \\
  E_{m+1}^{U_2}(n)= V_{m+1}^{U_2}(n)\cup U_2(n-1)E_{m}^{U_2}(n-1),
\end{eqnarray*}
and
\begin{align*}
  &  D_{m+1}(n) =  \mathrm{dist}_{H_1}(E_{m+1}^{U_1}(n),E_{m+1}^{U_2}(n)) \\
   &  \leq \mathrm{dist}_{H_1}(V_{m+1}^{U_1}(n),V_{m+1}^{U_2}(n)) + \mathrm{dist}_{H_1}(U_1(n-1)E_{m}^{U_1}(n),E_{m}^{U_2}(n))\\
    & \leq C\sup_{l\in(0,m+1]}\|U_1(n-l)-U_2(n-l)\|_{\mathbb{U}},
\end{align*}
where the constant $C$ depends only on $\mu,K$.

Then the assertion \eqref{hpf4} is valid for any $n\in\mathbb{Z}$ and $k\in\mathbb{N}$. Since the relationship between $E_k^{U_1}(n)$ and $E_k^{U_2}(n)$ is symmetric, then we can obtain the same estimate for $\mathrm{dist}_{H_1}(E_k^{U_2}(n),E_k^{U_1}(n))$. Thus the proof of this lemma is complete.
\end{proof}
Then by repeating the proof of \textbf{Theorem 3.1} word for word in \cite{Zelik2005}, we can prove the H\"{o}lder continuity \eqref{con4}.

\begin{remark}
In the above argument, the bounded set \( B \) we have chosen is time-dependent, which is appropriate for non-autonomous dynamical systems, following the approach in \cite{Carvalho2013a}. Based on this, we construct the intermediate sets \(\{V_k(n)\}\) and \(\{E_k(n)\}\) to facilitate the construction of the exponential attractor. We not only demonstrate that these sets pullback attract all bounded sets, but also establish that the pullback attraction rate is exponential.
\end{remark}

\section{Uniform Exponential Attractors for the damped hyperbolic equation with sub-quintic growth}
\noindent

In this concluding section, we apply the constructive method we prposed in Theorem \ref{th1} to prove the stability of the non-autonomous dynamical system generated by the weakly damped hyperbolic equation with sub-quintic growth in a bounded domain. The aim of this section is to obtain the existence of uniform exponential attractors for the non-autonomous system and prove that the uniform attractors are H\"{o}lder continuous with respect to the varying external forces.

\subsection{Setting of the problem}
\noindent

Consider the following damped  hyperbolic equations in a smooth bounded domain $\Omega\subset\mathbb{R}^3$:
\begin{equation} \label{dw}
\begin{cases}
u_{tt} + u_t  -\Delta u = f(t,u)+g(x), \quad x\in\Omega,~t\in\mathbb{R}, \\
u(x,t) = 0,\quad   x\in\partial\Omega,~t\in\mathbb{R}, \\
u(x,\tau) = u_0(x),~u_t(x,\tau)=u_1(x),\quad   x\in\Omega.
\end{cases}
\end{equation}

We assume that the nonlinear term $f$ satisfies the following assumptions:
\begin{itemize}
  \item [1.] $f\in C^1(\mathbb{R})$ or local Lipschitz continuous.
  \item [2.] Growth condition:
  \begin{equation}\label{fg}
    |f(u)|\leq C (1+|u|^{5-\kappa}),
  \end{equation}
  where the constants $C>0$ and $0<\kappa<1$.
  \item [3.] Sign condition:
  \begin{equation}\label{sign}
    \limsup_{|u|\to\infty}\frac{f(u)}{u}<\lambda_1,
  \end{equation}
where $\lambda_1$ is the first eigenvalue of $-\Delta$ defined in $\Omega$ with homogeneous boundary condition.
\end{itemize}

In this section, we consider the solution of problem \eqref{dw} to be the Shatah-Struwe solution as mentioned in the introduction. The Shatah-Struwe solution was proposed to address the issue of non-uniqueness of weak solutions for the supercritical damped  hyperbolic equation. It ensures uniqueness under the extra estimates of condition of space-time regularity for energy weak solutions. Its definition is as follows:
\begin{definition}[\cite{Kalantarov2016}]
  A weak solution $u(t+\tau),$ $t\in[\tau,\tau+T]$ is called a Shatah-Struwe solution of Cauchy problem \eqref{dw} if the following additional regularity holds:
  \begin{equation}
    u\in L^4(\tau;\tau+T;L^{12}(\Omega)).
  \end{equation}
  \end{definition}

We also introduce an embedding inequality for space-time estimates as provided in \cite{Kalantarov2016}.

  \begin{lemma}[\cite{Kalantarov2016}]
  Assume $\Omega$ is a bounded open set, for all $r\in[0,1],$ if \\
  $v\in {L^{\frac{4}{r}}\big(0,T;L^{\frac{12}{2-r}}(\Omega)\big)},{L^4(0,T;L^{12}(\Omega))},{L^\infty(0,T:H^1(\Omega))},$ one always has the following interpolation inequality:
  \begin{equation}\label{inperpolation}
    \|v\|_{L^{\frac{4}{r}}\left(0,T;L^{\frac{12}{2-r}}(\Omega)\right)}\leq C\|v\|^r_{L^4(0,T;L^{12}(\Omega))} \|v\|^{1-r}_{L^\infty(0,T:H^1(\Omega))}.
  \end{equation}
\end{lemma}

Under the above assumptions, it is well known that for every $(u_1,u_2)\in H^1_0\times L^2,$ equation \eqref{dw} possesses a unique Shatah-Struwe solution $u(t),$ $t\geq \tau$ (see \cite{Blair2009,Kalantarov2016}). Thus equation \eqref{dw} generates a dynamical process $\{U_f(t,\tau),t,\tau\in\mathbb{R},t\geq\tau\}$ in the phase space $H^1_0\times L^2$ by
\begin{equation*}
  U_f(t,\tau)(u_1,u_2)=(u(t),u_t(t)),
\end{equation*}
where $u(t)$ is the Shatah-Struwe solution to \eqref{dw} with initial data $$(u(\tau,\cdot),u_t(\tau,\cdot))=(u_0,u_1).$$
From Theorem 3.4 in \cite{Lidd2021}, for every $r_0\geq 0,$ there exists a positive constant $R_0$ and $t_0=t_0(r_0)$ such that for every $T>t_0$ and every Shatah-Struwe solution $u(t)$ to \eqref{dw} with $\|\nabla u_0\|^2+\|u_1\|^2\leq r_0^2,$ we have the following estimate:
\begin{equation}
  \|u_t(T)\|^2+\|\nabla u(T)\|^2\leq R_0^2.
\end{equation}
Then we can define the absorbing set for $U_f$ by
\begin{equation}
  B(t)\equiv B=\left\{(u,v)\in H^1_0\times L^2 \Big| \|u_t\|^2+\|\nabla u\|^2 \leq R_0^2 \right\}.
\end{equation}

From Theorem 3.5 in \cite{Lidd2021}, we can conclude that the Strichartz estimate for the Shatah-Struwe solution $u(t)$ of equation \eqref{dw} is independent of the specific choice of $f$  and depends only on the constant $C$ in \eqref{fg}.

To obtain the squeezing and the smoothing property for the dynamical system, we split it into two parts:
\begin{equation}
  U(t,\tau)=S(t,\tau)+C(t,\tau),
\end{equation}
where $S(t,\tau)(u_0,u_1)=(v(t),v_t(t))$ solves
\begin{equation}\label{equ_contraction}
  \begin{cases}
    v_{tt}+v_t-\Delta v=0,\quad  x\in\Omega,\\
    v(\cdot,\tau)=u_0,~v_t(\cdot,\tau)=u_1,\\
    v|_{\partial\Omega}=0,
  \end{cases}
\end{equation}
and the reminder $C$ satisfies
\begin{equation}\label{equ_smoothing}
  \begin{cases}
    w_{tt}+w_t-\Delta w=f(t,u)+g(x), \quad x\in\Omega,\\
    w(\cdot,\tau)=0,~w_t(\cdot,\tau)=0,\\
    w|_{\partial\Omega}=0,
  \end{cases}
\end{equation}

The squeezing property is an essential part for a non-compact dynamical system. It shows the part which is not regular(or compact) will decay to zero.
\begin{lemma}[Squeezing]\label{lemma_1}
  Let the above assumptions hold. $B$ is the absorbing set for $U(t,\tau).$ Assume $(v,v_t)$ and $(\tilde{v},\tilde{v}_t)$ are two Shatah-Struwe solutions to \eqref{equ_contraction} with initial $(u_0,u_1)$ and $(\tilde{u}_0,\tilde{u}_1)$ respectively. If $(u_0,u_1)$, $(\tilde{u}_0,\tilde{u}_1)\in B,$ then the difference $v-\tilde{v}$ satisfies
  \begin{align}\label{contration}
          &\|v-\tilde{v}\|^2_{H^1_0} +\|v_t-\tilde{v}_t\|^2_{L^2}+\|v-\tilde{v}\|_{L^4(t,t+1;L^{12}(\Omega))} \nonumber\\
          \leq &ce^{-\alpha_1(t-\tau)}\left(\|u_0-\tilde{u}_0\|^2_{H^1_0} +\|u_1-\tilde{u}_1\|^2_{L^2}\right).
  \end{align}

\end{lemma}
\begin{proof}
  Denote $y=v-\tilde{v},$ then it solves
  \begin{equation}
    y_{tt}+y_t-\Delta y=0.
  \end{equation}
  Multiply $y_t+\delta y$ on the above equation and integrate over $\Omega,$ it follows
  \begin{equation*}
    \frac{\mathrm{d}}{\mathrm{d}t}\left(\frac{1}{2}\|y_t\|^2+\frac{1}{2}\|\nabla y\|^2+\delta(y,y_t)\right) +(1-\delta)\|y_t\|^2 +\delta\|\nabla y\|^2 +\delta(y,y_t)=0.
  \end{equation*}
  Thus one can derive
  \begin{equation*}
    \frac{\mathrm{d}}{\mathrm{d}t}\left(\frac{1}{2}\|y_t\|^2+\frac{1}{2}\|\nabla y\|^2+\delta(y,y_t)\right) +\sigma_1\left(\frac{1}{2}\|y_t\|^2+\frac{1}{2}\|\nabla y\|^2+\delta(y,y_t)\right)\leq 0,
  \end{equation*}
  where the positive constant $\sigma_1$ depends on $\delta,$ but is independent of $y,y_t,t,\tau.$ Making use of Gronwall's inequality and Schwartz inequality, with a small enough $\delta,$ one has
  \begin{equation}\label{contro_energy}
    \|y_t(t)\|^2+\|\nabla y(t)\|^2\leq ce^{-\sigma_1(t-\tau)}\left(\|u_0-\tilde{u}_0\|^2_{H^1_0} +\|u_1-\tilde{u}_1\|^2_{L^2}\right).
  \end{equation}
  From Theorem 1.1 in \cite{Blair2009}, the local(with respect to time) Strichartz estimate holds for $y.$ Following a analogous discussion in \cite{Kalantarov2016}, one could obtain the following Strichartz estimate: for every $t\in\mathbb{R},$
  \begin{equation}\label{contro_Strichartz}
    \|v\|_{L^4(t,t+1;L^{12}(\Omega))}\leq ce^{\sigma_2(t-\tau)}\left(\|u_0-\tilde{u}_0\|_{H^1_0} +\|u_1-\tilde{u}_1\|_{L^2}\right),
  \end{equation}
  where $\sigma_2>0$ is a constant.
  Combining \eqref{contro_energy} and \eqref{contro_Strichartz}, we can prove energy decay estimate \eqref{contration}.
\end{proof}

From \cite{Carvalho2013,Chang2020,Chueshov2015}(for detail, Section 6 in \cite{Chang2020} and Chapter 3 in \cite{Chueshov2015}), it is shown that the squeezing and smoothing properties are sufficient to prove the finiteness of fractal dimensions. As discussed in Section 3, the smoothing property also ensures the finite covering of the absorbing set, which allows for the construction of exponential attractors.

The following lemma will demonstrate the smoothing property for the dynamical system generated by \eqref{dw}.
\begin{lemma}[Smoothing]\label{lemma_2}
  Let the assumptions in the above lemma hold. $B$ is the absorbing set for $U(t,\tau).$ For any $(u_0,u_1),(\tilde{u}_0,\tilde{u}_1)\in B,$ $w$ is the corresponding Shatah-Struwe solution for equation \eqref{equ_smoothing} with $(u(\tau,x),u_t(\tau,x)=(u_0,u_1))$ and $(\tilde{u}(\tau,x),\tilde{u}_t(\tau,x)=(\tilde{u}_0,\tilde{u}_1)).$ Then there exist positive constant $c$ and $\alpha_2$ such that the difference $w-\tilde{w}$ satisfies
  \begin{equation}
    \begin{aligned}
       \|w-\tilde{w}\|^2_{H^1_0} +\|w_t-\tilde{w}_t\|^2_{L^2}+\|w-\tilde{w}\|_{L^4(t,t+1;L^{12}(\Omega))}\\
      \leq  ce^{-\alpha_2(t-\tau)}\left(\|w_0-\tilde{w}_0\|^2_{H^1_0} +\|w_1-\tilde{w}_1\|^2_{L^2}\right).
    \end{aligned}
  \end{equation}
\end{lemma}
\begin{proof}
  Denote $y=w-\tilde{w},$ then it solves
  \begin{equation}\label{equ_y}
    y_{tt}+y_t-\Delta y+f(w,t)-f(\tilde{w},t)=0.
  \end{equation}
  From the interpolation inequality,
  \begin{equation*}
    \|f(w)-f(\tilde{w})\|_{H^\delta}\leq \|f(w)-f(\tilde{w})\|^{1-\theta}_{W^{1,6/5}}\|f(w)-f(\tilde{w})\|^\theta_{L^{10/5-\kappa}}.
  \end{equation*}
  Hence we can derive the $H^\delta$-norm if the right hand side of the above inequality is bounded.
  \begin{equation*}
    \begin{aligned}
      \|f(w)-f(\tilde{w})\|_{W^{1,6/5}}&\leq   c\left(\int_\Omega|(1+|w|^{3-\kappa}+|\tilde{w}|^{3-\kappa})(\nabla w+\nabla \tilde{w})y|^{\frac{6}{5}}\mathrm{d}x\right)^{\frac{5}{6}} \\
       & +c\left(\int_\Omega\left|(1+|w|^{4-\kappa}+|\tilde{w}|^{4-\kappa})|\nabla y|\right|^{\frac{6}{5}}\mathrm{d}x\right)^{\frac{5}{6}}\\
    &  = I_1+I_2.
    \end{aligned}
  \end{equation*}
  Due to H\"older's inequality,
  \begin{equation*}
    \begin{aligned}
      I_1&\leq  c\left(1+\|w\|^3_{L^{12}}+\|\tilde{w}\|^3_{L^{12}}\right)(\|\nabla w\|_{L^6}+\|\nabla \tilde{w}\|_{L^6})\|y\|_{L^{\frac{12}{5}}}\\
      &\leq  c\left(1+\|w\|^3_{L^{12}}+\|\tilde{w}\|^3_{L^{12}}\right)(\|w\|_{H^2}+\|\tilde{w}\|_{H^2})\|y\|_{H^1_0},
    \end{aligned}
  \end{equation*}
  and
  \begin{equation*}
    I_2\leq c\left(1+\|w\|^4_{L^{12}}+\|\tilde{w}\|^4_{L^{12}}\right)\|y\|_{H^1_0}.
  \end{equation*}
  Hence
  \begin{equation}
    \|f(w)-f(\tilde{w})\|_{W^{1,6/5}}\leq c\left(1+\|w\|^4_{L^{12}}+\|\tilde{w}\|^4_{L^{12}}\right)\|y\|_{H^1_0}.
  \end{equation}
  On the other hand,
  \begin{equation}
    \begin{aligned}
      \|f(w)-f(\tilde{w})\|_{L^{10/5-\kappa}}&\leq  \|f(w)\|_{L^{10/5-\kappa}}+\|f(\tilde{w})\|_{L^{10/5-\kappa}}\\
     & \leq  C\left(1+\|w\|^{5-\kappa}_{L^5(t,t+1;L^{10})} +\|\tilde{w}\|^{5-\kappa}_{L^5(t,t+1;L^{10})}\right)\\
     & \leq  C.
    \end{aligned}
  \end{equation}
  Hence we can conclude that
  \begin{equation}
    \|f(w)-f(\tilde{w})\|_{H^\delta}\leq c\left(1+\|w\|^4_{L^{12}}+\|\tilde{w}\|^4_{L^{12}}\right)\|y\|_{H^1_0}.
  \end{equation}
  Step 3: applying $(-\Delta)^{\frac{\delta}{2}}$ to \eqref{equ_y}, and multiplying $(-\Delta)^{\frac{\delta}{2}}y_t+\delta(-\Delta)^{\frac{\delta}{2}}y,$ one has
  \begin{equation*}
    \begin{aligned}
       \frac{\mathrm{d}}{\mathrm{d}t}\left(\frac{1}{2}\|(-\Delta)^{\frac{\delta}{2}}y_t\|^2 +\frac{1}{2}\|y\|^2_{H^{1+\delta}} +\delta((-\Delta)^{\frac{\delta}{2}}y_t,(-\Delta)^{\frac{\delta}{2}}y)\right)\\
       +(1-\delta)\|(-\Delta)^{\frac{\delta}{2}}y_t\|^2 +\delta\|y\|^2_{H^{1+\delta}}+\delta((-\Delta)^{\frac{\delta}{2}}y_t,(-\Delta)^{\frac{\delta}{2}}y)\\
      = \left((-\Delta)^{\frac{\delta}{2}}f(\tilde{w})-(-\Delta)^{\frac{\delta}{2}}f(w), (-\Delta)^{\frac{\delta}{2}}y_t+\delta(-\Delta)^{\frac{\delta}{2}}y\right).
    \end{aligned}
  \end{equation*}
  Denote $E_y=\frac{1}{2}\|(-\Delta)^{\frac{\delta}{2}}y_t\|^2 +\frac{1}{2}\|y\|^2_{H^{1+\delta}} +\delta((-\Delta)^{\frac{\delta}{2}}y_t,(-\Delta)^{\frac{\delta}{2}}y),$ then due to H\"older's inequality and the growth condition \eqref{fg}, one could obtain
  \begin{equation*}
    2\sqrt{E+\beta}\frac{\mathrm{d}}{\mathrm{d}t}\sqrt{E+\beta}+2\sigma_3(E+\delta)\leq c\left(1+\|w\|^4_{L^{12}}+\|\tilde{w}\|^4_{L^{12}}\right)\|y\|_{H^1_0}+2\sigma_3\delta.
  \end{equation*}
  Then taking advantage of Gronwall's inequality and passing $\beta\to 0^+,$ one has
  \begin{equation}
    \begin{aligned}
      \|(-\Delta)^{\frac{\delta}{2}}y_t(t)\|^2 +\|y(t)\|^2_{H^{1+\delta}}&\leq  C\int^t_\tau e^{-\sigma_3(t-s)}\left(1+\|w\|^4_{L^{12}}+\|\tilde{w}\|^4_{L^{12}}\right)\|y\|_{H^1_0}\mathrm{d}s\\
    &  \leq  C\|y\|_{L^\infty(\tau,t;H^1_0)}.
    \end{aligned}
  \end{equation}
  Multiplying $y_t$ on \eqref{equ_y} and doing a similar discussion, it follows that
  \begin{equation}
    \|y_t(t)\|^2+\|y(t)\|^2_{H^1_0}\leq C\|y_t(\tau)\|^2+\|y(\tau)\|^2_{H^1_0}.
  \end{equation}
  Consequently,
  \begin{equation*}
    \left\|(-\Delta)^{\frac{\delta}{2}}y_t(t)\right\|^2 +\|y(t)\|^2_{H^{1+\delta}}\leq C\|y_t(\tau)\|^2+\|y(\tau)\|^2_{H^1_0}.
  \end{equation*}
  Thus we complete the proof.
\end{proof}

Moreover, the following natural Lipschitz continuity holds for solutions of equation \eqref{dw} with differential external forces $g,\bar{g}\in L^1(\tau,t;L^2(\Omega))$ (see \cite{Zelik2005}).
\begin{lemma}\label{Holder_continuity}
  Let the assumptions in Lemma \ref{lemma_1} hold, $u$ and $\bar{u}$ are two Shatah-Struwe solutions to \eqref{dw} with $g$ being replaced by $g$ and $\bar{g}$ respectively. Their initial data are $(u_0,u_1)$ and $(\bar{u}_0,\bar{u}_1).$ If $(u_0,u_1)=(\bar{u}_0,\bar{u}_1),$ it holds that
  \begin{equation}\label{diff}
    \|u_t-\bar{u}_t\|_{L^2}+\|u-\bar{u}\|_{H^1_0}+\|u-\bar{u}\|_{L^4(\tau,t;L^{12}(\Omega))}\leq C\int^t_\tau e^{L(t-s)}\|g-\bar{g}\|_{L^2}\mathrm{d}s.
  \end{equation}
\end{lemma}
\begin{proof}
  Denote $y=u-\bar{u},$ then it solves
  \begin{equation}
    y_{tt}+y_t-\Delta y=f(u)-f(\bar{u})+g-\bar{g}.
  \end{equation}
  Multiplying $y_t,$ due to H\"older's inequality and Young's inequality, it yields that
  \begin{equation*}
    \begin{aligned}
      &\frac{\mathrm{d}}{\mathrm{d}t}\left(\frac{1}{2}\|y_t\|^2+\frac{1}{2}\|\nabla y\|^2\right)\\
      \leq & C\int_\Omega (1+|u|^{4-\kappa}+|\bar{u}|^{4-\kappa})yy_t\mathrm{d}x +\|g-\bar{g}\|\|y_t\|\\
     \leq   &C\left(1+\|u\|^{4}_{L^4(\tau,t;L^{12})}+\|\bar{u}\|^{4}_{L^4(\tau,t;L^{12})}\right)\|y\|_{L^6}\|y_t\|  +\|g-\bar{g}\|\|y_t\|\\
      \leq & L\left(\frac{1}{2}\|y_t\|^2+\frac{1}{2}\|\nabla y\|^2\right) +C\|g-\bar{g}\|^2.
    \end{aligned}
  \end{equation*}
  Then applying Gronwall's inequality, one has
  \begin{equation}\label{step1}
    \|y_t\|^2+\|\nabla y\|^2 \leq e^{L(t-\tau)}(\|y_t(\tau)\|^2+\|\nabla y(\tau)\|^2) +C\int^t_\tau e^{L(t-s)}\|g-\bar{g}\|^2\mathrm{d}s.
  \end{equation}
  From Theorem 1.1 in \cite{Burq2008} and Theorem 1.1 in \cite{Blair2009}, it yields
  \begin{equation}
    \begin{aligned}
       &\|y_t\|_{L^4(\tau,t;L^{12}(\Omega))}
      \\
      \leq  &C_t\left\{\|y_t(\tau)\|+\|\nabla y(\tau)\|)+
      \|g-\bar{g}\|_{L^1(0,T;L^2(\Omega))}+\|f(u)-f(\bar{u})\|_{L^1(0,T;L^2(\Omega))}\right\}.
    \end{aligned}
  \end{equation}
From the proof of Lemma \ref{lemma_2} and proof of Corollary 4.3 in \cite{Kalantarov2016},
  \begin{equation}
    \begin{aligned}
      &\|f(u)-f(\bar{u})\|_{L^1(\tau,t;L^2)}\\
      \leq & \|y\|_{L^\infty(\tau,t;H^{(1-\frac{\kappa}{4})})} \left(1+\|u\|^{4}_{L^4(\tau,t;L^{12})}+\|\bar{u}\|^{4}_{L^4(\tau,t;L^{12})}\right)\\
      \leq & C(\|y_t(\tau)\|^2+\|\nabla y(\tau)\|^2) +C\int^t_\tau e^{L_1(t-s)}\|g-\bar{g}\|^2\mathrm{d}s.\\
    \end{aligned}
  \end{equation}
  Then we can end up with \eqref{diff}.
\end{proof}

 The discussion in Section 3 allows us to construct a family of exponential attractors for the Cauchy problem \eqref{dw}. The squeezing and smoothing properties ensure that the long-term behavior of the system is close to that of a finite-dimensional system. 
 The choice of absorbing sets and the construction method of exponential attractors in this paper ensure that the resulting exponential attractors possess semi-invariance.
 Moreover, the exponential attraction can be obtained from Lemma \ref{fuzhu1}. In conclusion, we can deduce the Hölder continuity of systems with respect to external force from  Lemma \ref{Holder_continuity} which we formulate as the following theorem:

\begin{theorem}
  Let \eqref{fg} and \eqref{sign} hold $g\in L^1(\tau,t;L^2)$ for all $\tau<t\in\mathbb{R}.$ Then for every external force $g,$ there is an exponential attractor $t\mapsto\mathcal{M}_g(t)$ of the dynamical process $U_g$ which satisfies the following properties:
  \begin{enumerate}
    \item The sets $\mathcal{M}_g$ are compact finite-dimensional subsets of $H^1_0(\Omega)\times L^2(\Omega)$:
    \begin{equation}
      \dim_F(\mathcal{M}_g(t))\leq C,~~t\in\mathbb{R},
    \end{equation}
    where $C$ is a positive constant.
    \item These sets are semi-invariant with respect to $U_g(t,\tau)$ and translation invariant with respect to time shifts:
        \begin{equation}
          U_g(t,\tau)\mathcal{M}_g(\tau)\subset\mathcal{M}_g(t), \mathcal{M}_g(t+s)=\mathcal{M}_{T_sg}(t),
        \end{equation}
        where $t,s,\tau\in\mathbb{R},\tau\leq t$ and $\{T_s,s\in\mathbb{R}\}$ is a group of temporal shifts, $(T_sg)(t):=g(t+s).$
    \item Theses sets possess a uniform exponential attraction property: there exist a positive constant $\alpha$ and a monotonic function $Q$ such that for every bounded set $B\subset H^1_0(\Omega)\times L^2(\Omega),$ we have
        \begin{equation}
          dist_{H^1_0\times L^2}\left(U_g(t+\tau,\tau)B,\mathcal{M}_g(t+\tau)\right)\leq Q(|B|)e^{-\alpha t}
        \end{equation}
        for all $\tau\in\mathbb{R}$ and $t\geq 0.$
    \item The map $g\mapsto\mathcal{M}_g(t)$ is H{\"o}lder continuous in the following sense:
         \begin{equation}
           dist_{H^1_0\times L^2}^{symm}(\mathcal{M}_{g_1}(t),\mathcal{M}_{g_2}(t))\leq C\left(\int^t_{-\infty}e^{\beta(t-s)}\|g_1(s)-g_2(s)\|^2_{L^2}\mathrm{d}x\right)^\vartheta,
         \end{equation}
         where $C,\beta$ and $\theta$ are positive constants and independent of $g_1,$ and $g_2.$
  \end{enumerate}
\end{theorem}


The results of this paper also indicate that exponential attractors are more robust than global attractors to some extent, as the stability results obtained in this paper cannot be achieved for global attractors under the same conditions.
In this section, we use the weakly damped  hyperbolic equation to illustrate the application of Theorem \ref{th1}. However, Theorem \ref{th1} is also applicable to other dissipative non-autonomous systems under certain conditions, including reaction-diffusion equations and the 2D Navier-Stokes equations and other dissipative equations.

\section*{Acknowledgments}
This work was supported by NSFC (Grant No. 12201142 and No. 12301297).


\end{document}